\documentclass[english,11pt]{article}

\usepackage[T1]{fontenc}
\usepackage{babel}
\usepackage{amsmath,amsfonts,amsthm}
\usepackage{pdfsync}
\usepackage{color}

\DeclareMathSymbol{\leqslant}{\mathalpha}{AMSa}{"36} 
\DeclareMathSymbol{\geqslant}{\mathalpha}{AMSa}{"3E} 
\renewcommand{\leq}{\;\leqslant\;}                   
\renewcommand{\geq}{\;\geqslant\;}                   

\newtheorem{Th}{Theorem}
\newtheorem{Le}[Th]{Lemma}
\newtheorem{Pro}[Th]{Proposition}

\newcommand{\cA}{\ensuremath{\mathcal A}}
\newcommand{\cB}{\ensuremath{\mathcal B}}

\newcommand{\cD}{\ensuremath{\mathcal D}}
\newcommand{\cE}{\ensuremath{\mathcal E}}

\newcommand{\cH}{\ensuremath{\mathcal H}}

\newcommand{\cP}{\ensuremath{\mathcal P}}

\newcommand{\cR}{\ensuremath{\mathcal R}}

\newcommand{\cX}{\ensuremath{\mathcal X}}

\newcommand{\bbD}{{\ensuremath{\mathbb D}} }
\newcommand{\bbE}{{\ensuremath{\mathbb E}} }

\newcommand{\bbN}{{\ensuremath{\mathbb N}} }

\newcommand{\bbP}{{\ensuremath{\mathbb P}} }

\newcommand{\bbR}{{\ensuremath{\mathbb R}} }

\newcommand{\Om}{\Omega}
\newcommand{\E}{\bbE}

\newcommand{\Ne}{\bbN^{\ast}}
\newcommand{\R}{\bbR}
\newcommand{\bbd}{\mathbf{d}}

\newcommand{\tN}{\tilde{N}}
\newcommand{\hOm}{\hat{\Om}}
\newcommand{\hcA}{\hat{\cA}}
\newcommand{\hbbP}{\hat{\bbP}}
\newcommand{\hE}{\hat{\E}}
\newcommand{\hw}{\hat{w}}
\newcommand{\crea}{\varepsilon^+}
\newcommand{\anni}{\varepsilon^-}
\newcommand{\sbbD}{\underline{\bbD}}

\newcommand{\LP}{\cH_{\cP}}
\newcommand{\LD}{{\cH}_{\bbD}}
\newcommand{\LDP}{{\cH}_{\bbD,\cP}}
\newcommand{\LDdP}{{\cH}_{\bbD\otimes{\bbd},\cP}}

\newcommand{\xs}{X_{s^-}}
\newcommand{\alu}{^{(\alpha ,u)}}

\newcommand{\bK}{\bar{K}}

\title{Dirichlet Forms for Poisson Measures and L\'evy Processes: The Lent Particle Method}\date{}
\author{Nicolas BOULEAU and Laurent DENIS}
\begin{document}
\maketitle
\date{}
\begin{abstract}
 \hspace{.3cm}We present a new approach to absolute continuity of laws of
Poisson functionals. The theoretical framework is that of local Dirichlet forms as a tool to study probability spaces. The method gives rise to a new explicit calculus that we show first on some simple examples : it consists in adding a particle and taking it back after computing the gradient. Then we apply it to SDE's driven by Poisson measure.
\end{abstract}

\section{Introduction}
In order to situate the method it is worth to emphasize some features of the Dirichlet forms approach with comparison to the Malliavin calculus which is generally better known among probabilists.

First the arguments hold under only Lipschitz hypotheses : for example the method applies to a stochastic differential equation with Lipschitz coefficients (cf.$\!$ our second lecture in this volume). Second a general criterion exists, (EID) the Energy Image Density property, (proved on the Wiener space for the Ornstein-Uhlenbeck form, still a conjecture in general cf.$\!$ Bouleau-Hirsch \cite{bouleau-hirsch2} but established in the case of random Poisson measures with natural hypotheses) which ensures the existence of a density for a $\mathbb{R}^d$-valued random variable. Third, Dirichlet forms are easy to construct in the infinite dimensional frameworks encountered in probability theory and this yields a theory of errors propagation through the stochastic calculus, especially for finance and physics cf.$\!$ Bouleau \cite{bouleau3}, but also for numerical analysis of PDE's and SPDE's cf.$\!$ Scotti \cite{scotti}.

Our aim is to extend, thanks to Dirichlet forms, the Malliavin calculus applied to the case of Poisson measures and SDE's with jumps. Let us recall that in the case of jumps, there are several ways for applying the ideas of Malliavin calculus. The works are based either on the chaos decomposition (Nualart-Vives
  \cite{nualart-vives}) and provide tools in analogy with the Malliavin calculus on Wiener space,
  but non-local (Picard \cite{picard}, Ishikawa-Kunita \cite{ishikawa-kunita}) or dealing with local operators acting on the size of the jumps using the expression
  of the generator on a sufficiently rich class and closing the structure, for instance by Friedrichs' argument (cf.$\!$  especially Bichteler-Gravereaux-Jacod \cite{bichteler-gravereaux-jacod} Coquio \cite{coquio} Ma-R\"ockner \cite{ma-rockner2}).
  
  We follow a way close to this last one. We will first expose the method from a practical point of view, in order to show how it acts on concrete cases. Then in a separate part we shall give the main elements of the proof of the main theorem on the lent particle formula. Eventually we will display several examples where the method improves known results.  Then, in the last section, we shall apply the lent particle method to SDE's  driven by a Poisson measure or a Lévy process. Complete details of the proofs and hypotheses for getting (EID) are published in \cite{bouleau4} and \cite{bouleau-denis1}.

  \section{The lent particle method}
  
  Consider a random Poisson measure as a distribution of points, and let us see a L\'evy process as defined by a Poisson measure, that is let us think on the {\it configuration space}. We suppose the particles live in a space (called the bottom space) equipped with a local Dirichlet form with carr\'e du champ and gradient. This makes it possible to construct a local Dirichlet form with carr\'e du champ on the configuration space (called the upper space). To calculate for some functional the Malliavin matrix  -- which  in the framework of Dirichlet forms becomes the carr\'e du champ matrix -- the method consists first in adding a particle to the system. The functional then possesses a new argument which is due to this new particle. We can compute the bottom-gradient of the functional with respect to this argument and as well its bottom carr\'e du champ. Then taking back the particle we have added does not erase the new argument of the obtained functional. We can integrate the new argument with respect to the Poisson measure and this gives the upper carr\'e du champ matrix -- that is the Malliavin matrix. This is the exact summary of the method.\\
  
\subsection{ Let us give  more details and notation.} 
  
    Let $(X,\cX
,\nu,\bbd,\gamma)$ be a local symmetric Dirichlet structure which
admits a carr\'e du champ operator. This means that $(X,\cX ,\nu )$ is a
measured space, $\nu$ is $\sigma$-finite
and the bilinear form
$ e [f,g]=\frac12\int\gamma [f,g]\, d\nu,$
is a local  Dirichlet form with domain $\bbd\subset L^2 (\nu )$
and carr\'e du champ $\gamma$ (cf Fukushima-Oshima-Takeda \cite{fukushima-oshima-takeda} and Bouleau-Hirsch \cite{bouleau-hirsch2}). $(X,\cX
,\nu,\bbd,\gamma)$ is called the bottom space.

Consider a Poisson random measure $N$ on $[0,+\infty[\times X$ with intensity measure $dt\times \nu$.

\noindent A Dirichlet structure may be constructed canonically on the probability space of this Poisson measure that we denote $(\Omega_1, \mathcal{A}_1, \mathbb{P}_1, \mathbb{D}, \Gamma)$. We call this space the upper space.

$\mathbb{D}$ is a set of functions in the domain of $\Gamma$, in other words a set of random variables which are functionals of the random distribution of points. The main result  is the following formula:

 For all $F\in\bbD$
\[\Gamma [F] =\int_0^{+\infty}\int_X \anni(\gamma [\crea F ])\, dN\]

\noindent in which $\crea$ and $\anni$  are the creation and annihilation operators.

Let us explain the meaning and the use of this formula on an example. \\
\subsection{First example.}

Let  $Y_t$ be a centered L\'evy process with  L\'evy measure $\nu$ integrating $x^2$. We assume that $\nu$ is such that a local Dirichlet structure may be constructed on $\mathbb{R}\backslash\{0\}$ with carr\'e du champ
$
\gamma[f]=x^2f^{\prime 2}(x).
$

The notion of gradient in the sense of Dirichlet forms is explained in \cite{bouleau-hirsch2} Chapter V. It is a linear operator with values in an auxiliary Hilbert space giving the carr\'e du champ by taking the square of the Hilbert norm. It is convenient to choose for the Hilbert space a space $L^2$ of a probability space.

Here we define a gradient  $\flat$ associated with $\gamma$ by choosing $\xi$ such that $\int_0^1\xi(r)dr=0$ and $\int_0^1\xi^2(r)dr=1$ and putting
$$f^\flat=xf^\prime(x)\xi(r).
$$
Practically $\flat$ acts as a derivation with the chain rule $(\varphi(f))^\flat=\varphi^\prime(f).f^\flat$ (for $\varphi\in\mathcal{C}^1\cap Lip$ or even only Lipschitz).

 $N$ is the  Poisson random measure associated with $Y$ with intensity $dt\times\sigma$ such that $\int_0^th(s)\;dY_s=\int{\bf 1}_{[0,t]}(s)h(s)x\tilde{N}(dsdx)$ for $h\in L^2_{loc}(\mathbb{R}_+)$.
 
We study the regularity of  $$ V=\int_0^t\varphi(Y_{s-})dY_s$$ where $\varphi$ is Lipschitz and $\mathcal{C}^1$.

$1^o$. First step. We add a particle $(\alpha,x)$ i.e. a jump to $Y$ at time $\alpha$ with size $x$ what gives

$$\varepsilon^+V=V+\varphi(Y_{\alpha-})x+\int_{]\alpha}^t(\varphi(Y_{s-}+x)-\varphi(Y_{s-}))dY_s
$$

$2^o$. $V^\flat=0$ since $V$ does not depend on $x$, and

$$(\varepsilon^+V)^\flat=\left(\varphi(Y_{\alpha-})x+\int_{]\alpha}^t\varphi^\prime(Y_{s-}+x)xdY_s\right)\xi(r)\quad$$

\noindent because $x^\flat=x\xi(r)$.\\

$3^o$. We compute

$$\gamma[\varepsilon^+V]=\int(\varepsilon^+V)^{\flat2}dr=(\varphi(Y_{\alpha-})x+\int_{]\alpha}^t\varphi^\prime(Y_{s-}+x)xdY_s)^2$$

$4^o$. We take back the particle what gives $\varepsilon^-\gamma[\varepsilon^+V]= (\varphi(Y_{\alpha-})x+\int_{]\alpha}^t\varphi^\prime(Y_{s-})xdY_s)^2$

and compute $\Gamma[V]=\int\anni\gamma[\varepsilon^+V]dN$ (lent particle formula) 
$$\Gamma[V]=\int\left(\varphi(Y_{\alpha-})+\int_{]\alpha}^t\varphi^\prime(Y_{s-})dY_s\right)^2x^2\;N(d\alpha dx)$$ 
$$
=\sum_{\alpha\leq t}\Delta Y_\alpha^2(\int_{]\alpha}^t\varphi^\prime(Y_{s-})dY_s+\varphi(Y_{\alpha-}))^2.
$$ where $\Delta Y_\alpha=Y_\alpha-Y_{\alpha-}$.

For real functional, (EID) is always true : $V$ possesses a density as soon as $\Gamma[V]>0$. Then the above expression may be used to discuss the strict positivity of $\Gamma[V]$ depending on the finite or infinite mass of $\nu$ cf.$\!$ \cite{bouleau-denis1} Example 5.2.\\ Before giving a   typical set of assumptions  that the Lévy  measure $\nu$ has to fulfill, let us explicit the (EID) property.
\subsection{ Energy Image Density property (EID).}

A Dirichlet form on  $L^2(\Lambda)$ ($\Lambda$ $\sigma$-finite) with carr\'e du champ $\gamma$
satisfies (EID) if, for any $d$ and all $U$
 with values in $\mathbb{R}^d$ whose components are in the domain of the form,  the image by $U$ of the measure with density with respect to $\Lambda$ the determinant of the carr\'e du champ matrix is absolutely continuous with respect to the Lebesgue measure i.e.
 $$ U_*[({\det}{{\gamma}}[U,U^t])\cdot
\Lambda ]\ll \lambda^d .$$
This property is true for the Ornstein-Uhlenbeck form on the Wiener space, and in several other cases cf. Bouleau-Hirsch \cite{bouleau-hirsch2}.
It was conjectured in 1986 that it were always true. It is still a conjecture.\\

It is therefore necessary to prove this property in the context of Poisson random measures. With natural hypotheses, cf.$\!$ \cite{bouleau-denis1} Parts 2 and 4,  as soon as  EID is true for the bottom space, 
 EID is true for the upper space.
Our proof uses a result of   Shiqi Song \cite{song}.\\

 \subsection{Main example of bottom structure in $\R^d$}\label{example}
 Let $(Y_t )_{t\geq 0}$ be a $d$-dimensional Lévy process, with L\'evy measure $\nu =k dx$. Under standard hypotheses, we have the following representation:
\[ Y_t =\int_0^t \int_{\R^d} u\tN (ds ,du),\]
where $\tN$ is a compensated Poisson measure with intensity $dt\times kdx$. In this case, the idea is to introduce an ad-hoc Dirichlet structure on $\R^d$\\
The following example gives a case of such a structure $(\bbd ,e)$ which satisfies all the required hypotheses and which is flexible enough to encompass  many cases:
\begin{Le}{\label{lemmed}}
 Let $r\in \Ne$, $(X,\cX )=(\R^r ,\mathcal{B}(\R^r))$ and $\nu =k dx$ where $k$ is non-negative and Borelian. We are given  $\xi=(\xi_{ij})_{1\leq i,j\leq r}$ an $\R^{r\times r}$-valued and
symmetric Borel function. We assume that there exist an open set $O\subset \R^r$ and a  function $\psi$ continuous on $O$ and null on $\R^r \setminus O$ such that
\begin{enumerate}
\item  $k>0$ on $O$ $\nu$-a.e. and is locally bounded on $O$.
\item  $\xi$ is locally bounded and  locally elliptic on $O$.
\item $k\geq \psi>0$ $\nu$-a.e. on $O$.
\item for all $i,j\in \{ 1,\cdots ,r\}$, $\xi_{i,j}\psi$ belongs to $H^1_{loc} (O)$.
\end{enumerate}
We denote by $H$ the subspace of functions $f\in L^2 (\nu )\cap L^1 (\nu )$ such that the restriction of $f$ to $O$ belongs to $C_c^{\infty} (O )$.
Then, the bilinear form defined by
\[\forall f,g\in H,\ e(f,g)=\sum_{i,j=1}^r \int_O \xi_{i,j}(x)\partial_i f(x)\partial_j g (x)\psi(x)\, dx\]
is closable in $L^2 (\nu)$. Its closure, $(\bbd ,e)$, is a local Dirichlet form on $L^2 (\nu)$ which admits a carr\'e du champ $\gamma$:
\[ \forall f\in \bbd,\  \gamma (f)(x)=\sum_{i,j=1}^r \xi_{i,j}(x)\partial_i f(x)\partial_j f(x)\frac{\psi(x)}{k(x)}.\]Moreover, it satisfies
 {property {\rm(EID)}} i.e. for any $d$ and for any $\mathbb{R}^d$-valued function $U$
whose components are in the domain of the form
$$ U_*[({\det}\gamma[U,U^t])\cdot
\nu ]\ll \lambda^d $$ where {\rm det} denotes the determinant
and $\lambda^d$ the Lebesgue measure on $(\R^d ,\cB (\R^d ))$.\\
\end{Le}
{\it Remark:} In the case of a L\'evy process, we shall apply this Lemma with $\xi$ the identity application. We shall often consider an open domain 
of the form $O=\{ x\in \R^d ; \ |x| < \varepsilon\}$ which means that we  "differentiate" only w.r.t. small jumps and hypothesis $3.$ means that we do not need to assume regularity on $k$ but only that $k$ dominates a regular function.\\

  \subsection{ Multivariate example.}

Consider as in the previous section, a centered  Lévy process  without gausssian part $Y$ such that its Lévy measure $\nu$ satisfies assumptions of lemma \ref{lemmed} (which imply $1+\Delta Y_s\neq 0$ a.s.) with $d=1$ and $\xi (x)=x^2$.\\
 We want to study the existence of density for the pair  $(Y_t ,\mathcal{E}xp(Y )_t)$
where $\mathcal{E}xp(Y )$ is the Dol\'eans exponential of $Y$.

$$\label{exp}\mathcal{E}xp(Y)_t=e^{Y_t}\prod_{s\leq t}(1+\Delta Y_s)e^{-\Delta Y_s}.$$

\noindent$1^0/$ We add a particle $(\alpha ,y)$ i.e. a jump to $Y$ at time $\alpha\leq t$ with size $y$ :
$$\varepsilon^+_{(\alpha,y)}(\mathcal{E}xp(Y)_t)=e^{Y_t+y}\prod_{s\leq t}(1+\Delta Y_s)e^{-\Delta Y_s}(1+y)e^{-y}=\mathcal{E}xp(Y)_t(1+y).$$

\noindent$2^0/$ We compute $\gamma [\varepsilon^+ \mathcal{E}xp(Y)_t](y)=(\mathcal{E}xp(Y)_t)^2 y^2 \frac{\psi(y)}{k(y)}$.

\noindent$3^0/$ We take back the particle :
\[ \varepsilon^- \gamma [\varepsilon^+ \mathcal{E}xp(Y)_t]=\left( \mathcal{E}xp(Y)_t (1+y)^{-1}\right)^2 y^2\frac{\psi(y)}{k(y)}\]
we integrate in  $N$  and that gives the upper carr\'e du champ operator (lent particle formula):
\[
\begin{array}{rl} \Gamma [\mathcal{E}xp(Y)_t]&=\int_{[0,t]\times\mathbb{R}}\left( \mathcal{E}xp(Y)_t (1+y)^{-1}\right)^2 y^2\frac{\psi(y)}{k(y)}N(d\alpha dy)\\

&=\sum_{\alpha \leq t}\left( \mathcal{E}xp(Y)_t (1+\Delta Y_{\alpha})^{-1}\right)^2 \frac{\psi(\Delta Y_{\alpha})}{k(\Delta Y_{\alpha})}\Delta Y_{\alpha}^2.
\end{array}\]
By a similar computation  the matrix $\underline{\underline{\Gamma}}$ of the pair $(Y_t ,\mathcal{E}xp(Y_t ))$ is given by
\[ \underline{\underline{\Gamma}}=\sum_{\alpha \leq t}\left(
                                \begin{array}{cc}
                                  1 &  \mathcal{E}xp(Y)_t (1+\Delta Y_{\alpha})^{-1}\\
                                  \mathcal{E}xp(Y)_t (1+\Delta Y_{\alpha})^{-1} & \left( \mathcal{E}xp(Y)_t (1+\Delta Y_{\alpha})^{-1}\right)^2 \\
                                \end{array}
                              \right) \frac{\psi(\Delta Y_{\alpha})}{k(\Delta Y_{\alpha})}\Delta Y_{\alpha}^2 .
\]
Hence under hypotheses implying (EID), such as those of Lemma \ref{lemmed}, the density of the pair $(Y_t ,\mathcal{E}xp(Y_t ))$ is yielded by the condition
$$\mbox{dim  } \mathcal{L}\left(\left(\begin{array}{c}
1\\
\mathcal{E}xp(Y)_t(1+\Delta Y_\alpha)^{-1}
\end{array}\right)\quad \alpha\in JT\right)=2
$$
where $JT$ denotes the jump times of $Y$ between 0 and $t$. 

Making this in details we obtain

{\it Let $Y$ be a real L\'evy process with infinite L\'evy measure with density dominating near  $0$ a positive  function  locally in $H^1$, then the pair $(Y_t,\mathcal{E}xp(Y)_t)$ possesses a density on  $\mathbb{R}^2$.}

\section{Demonstration of the lent particle formula}
\subsection{The construction.}

\noindent Let us recall that $(X,\cX
,m,\bbd,\gamma)$ is a local Dirichlet structure with  carr\'e du champ called the bottom space, $m$ is  $\sigma$-finite
and the bilinear form $e [f,g]=\frac12\int\gamma [f,g]\, d m$
is a local Dirichlet form with domain $\bbd\subset L^2 ( m )$
and with carr\'e du champ $\gamma$.
For all $x\in X$, $\{ x\}$ is supposed to belong to $\cX$,  $ m$ is diffuse.
The associated generator is denoted   $a$, its domain is $\cD (a)\subset \bbd$. 

We consider a random Poisson measure  $N$,
on $(X,\cX , m )$ with intensity $m$. It is defined on $(\Om,
\cA, \bbP )$ where $\Om$ is the configuration space of countable sums of Dirac masses on  $E$, $\cA$
is the $\sigma$-field generated by  $N$ and   $\bbP$ is the law of  $N$.

 $(\Om, \cA ,\bbP)$
is called  the upper space. The question is to construct a Dirichlet structure on the upper space, induced "canonically" by the Dirichlet structure of the bottom space.

This question is natural by the following interpretation. The bottom structure may be thought as the elements for the description  of a single particle moving according to a symmetric Markov process associated with the bottom Dirichlet form. Then considering an infinite family of independent such particles with initial law given by  $(\Om, \cA ,\bbP)$ shows that a Dirichlet structure can be canonically considered on the upper space (cf. the introduction of \cite{bouleau-denis1} for different ways of tackling this question).\\

Because of some formulas on functions of the form  $e^{iN(f)}$ related to the Laplace functional, we consider the space of test functions $\cD_0$ to be the set of elements in $L^2 (\bbP )$
which are the linear combinations of variables of the form
$e^{i\tN (f)}$ with $f\in \left( \cD (a)\otimes L^2 (dt)\right)\bigcap L^1 (\nu\times dt)$, recall that $\tN =N-dt\times\nu$.\\
If $U=\sum_p \lambda_p e^{i\tN (f_p)}$ belongs to $\cD_0$, we put
\begin{equation}{\label{214}}
\tilde{A}_0 [U]=\sum_p \lambda_p e^{i\tN (f_p)}(i\tN ({a}[f_p
])-\frac12 N (\gamma [f_p])),\end{equation} where, in a natural way
, if $f(x,t)=\sum_l u_l (x) \varphi_l (t) \in \cD (a)\otimes
L^2 (dt)$
\[ {a} [f]=\sum_l a[u_l ]\varphi_l \makebox{ and } \gamma [f]=\sum_l \gamma [u_l ]\varphi_l.\]In order to show that
$A_0$  is uniquely defined and is the generator of a Dirichlet form satisfying the required properties, starting from a gradient of the bottom structure we construct a gradient for the upper structure defined first on the test functions. Then we show that this gradient does not depend on the form of the test function and this allows to extend the operators thanks to Friedrichs' property yielding the closedness of the upper structure.\\

\subsection{ The bottom gradient.}

We suppose the space  $\bbd$ separable, then there exists a gradient for the bottom space  :
There is a separable Hilbert space and a linear map  $D$ from  $\bbd$ into $L^2 (X, m;H)$ such that $\forall u\in \bbd$, $\| D[u ]\|^2_H =\gamma[u]$, then necessarily

 - If
$F:\R\rightarrow \R$ is Lipschitz  then
$\forall u\in\bbd$,

$\ D[F\circ u]=(F'\circ u )Du,$

- If $F$ is $\mathcal{C}^1$ and Lipschitz from $\R^d$ into
$\R$  then

$\ D[F\circ
u]=\sum_{i=1}^d (F'_i \circ u ) D[u_i ]\;\; \forall u=(u_1 ,\cdots ,u_d) \in \bbd^d.$\\

We take for $H$
a space $L^2 (R,\cR ,\rho)$ where $(R,\cR ,\rho)$  is a probability space s.t.   $L^2 (R,\cR
 ,\rho)$ is infinite dimensional. The gradient $D$ is denoted  by $\flat$:
 \[\forall u\in\bbd,\ Du=u^{\flat} \in L^2 (X\times R ,\cX \otimes \cR
 , m\otimes\rho).\]
  Without loss of generality, we assume moreover that the operator
 $\flat$ takes its values in the orthogonal space of $1$ in $L^2
 (R,\cR ,\rho)$. So that we have
$$\label{231}
 \forall u\in\bbd ,\ \int u^{\flat}d\rho =0\ \ \nu\mbox{-}a.e.
$$

\subsection{ Candidate gradient for the upper space.}

Then, we introduce the creation operator (resp. annihilation operator) which consists in adding (resp. removing if necessary) a jump at time $t$ with size $u$:
\[\begin{array}{l}
\crea_{(t ,u)} (w_1)=w_1{\bf 1}_{\{ (t ,u)\in supp
\, w_1\}}+(w_1+\varepsilon_{(t ,u)}\}) {\bf 1}_{\{ (t ,u)\notin supp \, w_1\}}\\
 \anni_{(t
,u)} (w_1)=w_1{\bf 1}_{\{ (t ,u)\notin supp \,
w_1\}}+(w_1-\varepsilon_{(t ,u)}\}) {\bf 1}_{\{ (t ,u)\in supp \,
w_1\}}.
\end{array}\]
In a natural way, we extend these operators to the functionals by
\[ \crea H(w_1,t ,u)=H(\crea_{(t ,u)} w_1,t, u)\quad\anni H(w_1,t ,u)=
H(\anni_{(t , u)}w_1,t,u).\]
\noindent Definition. {\it  For $F\in\cD_0$, we define the pre-gradient
\[ F^\sharp =\int_0^{+\infty}\int_{X\times R} \anni((\crea F)^{\flat})\, dN\odot\rho ,\]
where $N\odot\rho$ is the point process $N$ ``marked'' by $\rho$} 

\noindent i.e. if $N$ is the family of marked points $(T_i, X_i )$, $N\odot\rho$ is the family  $(T_i, X_i,r_i)$ where the $r_i$ are new independent random variables mutually independent and identically distributed with law $\rho$, defined on an auxiliary probability space $(\hOm, \hcA,\hbbP )$. So $N\odot\rho$ is a Poisson random measure on $[0,+\infty[\times X\times R$.\\

\subsection{ Main result.}

The above candidate may be shown to extend in a true gradient for the upper structure. The argument is based on the extension of the pregenerator  $A_0$ thanks to Friedrichs' property (cf. for instance \cite{bouleau-hirsch2} p. 4) : $A_0$ is shown to be well defined on $\cD_0$ which is dense, $A_0$ is negative and symmetric and therefore possesses a selfadjoint extension. Before stating the main Theorem, let us introduce some notation. We denote by $\sbbD$ the
completion of $\cD_0\otimes L^2 ([0,+\infty[,dt)\otimes \bbd$ with
respect to the norm
\begin{eqnarray*}\| H\|_{\sbbD}\!\!&=&\!\!\left( \E\int_0^{\infty}\!\!\!\int_{ X }\anni(\gamma
[H])(w,t,u)N(dt,du)\right)^{\frac12} \!+\E\int_0^{\infty}\!\!\!\int_X
(\anni|H|)(w,t,u)\eta (t,u)N(dt,du)\\
&=&\!\!\left( \E\int_0^{\infty}\!\!\!\int_{ X }\gamma
[H](w,t,u)\nu(du)dt\right)^{\frac12} \!+\E\int_0^{\infty}\!\!\!\int_X
|H|(w,t,u)\eta (t,u)\nu (du) dt,
\end{eqnarray*} where $\eta$ is a fixed
positive function in $L^2 (\R^+ \times X ,dt\times d\nu)$.\\

\noindent Theorem. { \it The formula
$$\label{31}
\forall F\in\bbD,\ F^\sharp =\int_0^{+\infty}\int_{X\times R} \anni((\crea F )^\flat)\,
dN\odot \rho ,$$ extends from $\cD_0$ to $\mathbb{D}$,  it is justified by the following  decomposition :
$$\hspace{-1.5cm} F\in\mathbb{D}\stackrel{\crea-I}{\mapsto} \varepsilon^+F-F\in\underline{\mathbb{D}}\stackrel{\anni((.)^\flat)}{\mapsto}\anni((\varepsilon^+F)^\flat)\in L^2_0(\mathbb{P}_N\times\rho)\stackrel{d(N\odot\rho)}{\mapsto} F^\sharp\in L^2(\mathbb{P}\times\hat{\mathbb{P}})$$where each operator is continuous on the range of the preceding one 
and where  $L^2_0
(\bbP_N \times \rho )$ is the closed set of elements  $G$ in $L^2
(\bbP_N \times \rho )$ such that $\int_R G d\rho=0$ $\bbP_N$-a.s.

Furthermore for all $F\in\bbD$
\begin{eqnarray}\label{formuleocc}\Gamma [F]=\hE (F^\sharp )^2 =\int_0^{+\infty}\int_X \anni\gamma [\crea F ]\, dN.\end{eqnarray}}

\noindent Let us explicit the steps of a typical calculation applying this theorem.

 Let  $H=\Phi(F_1,\ldots,F_n)$ with $\Phi\in\mathcal{C}^1\cap Lip(\mathbb{R}^n)$ 
and  $F=(F_1,\ldots,F_n)$ with $F_i\in\bbD$, we have :
$$\begin{array}{rrlll}

a)&\gamma[\crea H]&=\sum_{ij}\Phi^\prime_i(\crea F)\Phi^\prime_j(\crea F)\gamma[\crea F_i,\crea F_j]&\mathbb{P}\times\nu\mbox{-a.e.}\\
&&&&\\

b)&\anni\gamma[\crea H]&=\sum_{ij}\Phi^\prime_i(F)\Phi^\prime_j(F)\anni\gamma[\crea F_i,\crea F_j]&\mathbb{P}_N\mbox{-a.e.}\\
&&&&\\

c)&\Gamma[H]=\int\anni\gamma[\crea H]dN&=\sum_{ij}\Phi^\prime_i(F)\Phi^\prime_j(F)\int\anni\gamma[\crea F_i,\crea F_j]dN&\mathbb{P}\mbox{-a.e.}
\end{array}$$
As we see above, a peculiarity of the  method comes from the fact that it involves, in the computation, successively mutually singular measures, 
such as measures $\mathbb{P}_N=\mathbb{P}(d\omega)N(\omega,dx)$ and $\mathbb{P}\times\nu$. This imposes some care in the applications.\\
Let us finally remark that the lent particle formula (\ref{formuleocc}) has been encountered   by some authors  as valid on a space of test functions (see e.g. \cite{privault2} before Prop 8), let us emphasize that in our case, it is  valid on the whole domain $\bbD$, this  is essential to apply the method to SDE's for example.

\section{Applications}

\subsection{ Sup of a stochastic process on $[0,t]$.}

\noindent The fact that the operation of taking the maximum is typically a Lipschitz operation makes it easy to apply the method.

Let $Y$ be a centered  L\'evy process as in \S 2.2. 
Let $K$ be a  c\`adl\`ag process independent of $Y$. 
 We put   $$H_s=Y_s+K_s.$$
 
\noindent Proposition. {\it  If $\sigma(\mathbb{R}\backslash\{0\})=+\infty$ and if $\mathbb{P}[\sup_{s\leq t}H_s=H_0]=0$, the random variable  $\sup_{s\leq t}H_s$ has a density.
}\\

As a consequence,  {\it any L\'evy process starting from zero and immediately entering  $\mathbb{R}_+^\ast$, whose L\'evy measure dominates a measure $\sigma$ satisfying  Hamza condition and infinite, is such that $\sup_{s\leq t}X_s$ has a density.}

Let us recall that the Hamza condition (cf. Fukushima and al.\cite{fukushima-oshima-takeda} Chapter 3) gives a necessary and sufficient condition of existence of a Dirichlet structure on $L^2(\sigma)$. Such a necessary and sufficient condition is only known in dimension one.\\

\subsection{ Regularity without H\"ormander.}

\noindent Consider the following SDE driven by a two dimensional Brownian motion
\begin{equation}\label{diffusion}\left\{\begin{array}{rl}
X^1_t&=z_1+\int_0^tdB^1_s\\
X^2_t&=z_2+\int_0^t2X^1_sdB^1_s+\int_0^tdB^2_s\\
X^3_t&=z_3+\int_0^tX^1_sdB^1_s+2\int_0^tdB^2_s.
\end{array}\right.
\end{equation}
This diffusion is degenerate and the H\"ormander conditions are not fulfilled. The generator is $A=\frac{1}{2}(U_1^2+U_2^2)+V$
and its adjoint $A^\ast=\frac{1}{2}(U_1^2+U_2^2)-V$ with  
$U_1=\frac{\partial}{\partial x_1}+2x_1\frac{\partial}{\partial x_2}+x_1\frac{\partial}{\partial x_3}$,
$U_2=\frac{\partial}{\partial x_2}+2\frac{\partial}{\partial x_3}$ and $V=-\frac{\partial}{\partial z_2}-\frac{1}{2}\frac{\partial}{\partial z_3}$. The Lie brackets of these vectors vanish and the Lie algebra is of dimension 2 :
the diffusion remains on the quadric of equation
$\frac{3}{4}x_1^2-x_2+\frac{1}{2}x_3-\frac{3}{4}t=C.$

\noindent Consider now the same equation driven by a L\'evy process :

$$\left\{\begin{array}{rl}
Z^1_t&=z_1+\int_0^tdY^1_s\\
Z^2_t&=z_2+\int_0^t2Z^1_{s_-}dY^1_s+\int_0^tdY^2_s\\
Z^3_t&=z_3+\int_0^tZ^1_{s_-}dY^1_s+2\int_0^tdY^2_s
\end{array}\right.
$$
under hypotheses on the L\'evy measure  such that the bottom space may be equipped with the carr\'e du champ operator 
$\gamma[f]=y_1^2f^{\prime 2}_1+y_2^2f^{\prime 2}_2$ satisfying the hypotheses yielding EID.
Let us apply in full details the lent particle method.
$$\mbox{ For }\alpha\leq t\qquad\varepsilon^+_{(\alpha,y_1,y_2)}Z_t=Z_t+\left(
\begin{array}{c}
y_1\\
2Y^1_{\alpha-}y_1+2\int_{]\alpha}^ty_1dY^1_s+y_2\\
Y^1_{\alpha-}y_1+\int_{]\alpha}^ty_1dY^1_s+2y_2
\end{array}\right)=Z_t+\left(
\begin{array}{c}
y_1\\
2y_1Y_t^1+y_2\\
y_1Y_t^1+2y_2
\end{array}\right),
$$
where we have used $Y^1_{\alpha-}=Y^1_\alpha$ because $\varepsilon^+$ send into $\mathbb{P}\times \nu$ classes. That gives
$$\gamma[\varepsilon^+Z_t]=\left(
\begin{array}{lcr}
y_1^2&y_1^22Y^1_t&y_1^2Y^1_t\\
id&y_1^24(Y^1_t)^2+y_2^2&y_1^22(Y^1_t)^2+2y_2^2\\
id&id&y_1^2(Y_t^1)^2+4y_2^2
\end{array}
\right)$$
and 
$$\varepsilon^-\gamma[\varepsilon^+Z_t]=\left(
\begin{array}{lcr}
y_1^2&y_1^22(Y^1_t-\Delta Y_\alpha^1)&y_1^2(Y^1_t-\Delta Y_\alpha^1)\\
id&y_1^24(Y^1_t-\Delta Y_\alpha^1)^2+y_2^2&y_1^22(Y^1_t-\Delta Y_\alpha^1)^2+2y_2^2\\
id&id&y_1^2(Y_t^1-\Delta Y_\alpha^1)^2+4y_2^2
\end{array}
\right),$$
where $id$ denotes the symmetry of the matrices. Hence
$$\Gamma[Z_t]=\sum_{\alpha\leq t}(\Delta Y_\alpha^1)^2
\left(
\begin{array}{lcr}
1&2(Y^1_t-\Delta Y_\alpha^1)&(Y^1_t-\Delta Y_\alpha^1)\\
id&4(Y^1_t-\Delta Y_\alpha^1)^2&2(Y^1_t-\Delta Y_\alpha^1)^2\\
id&id&(Y_t^1-\Delta Y_\alpha^1)^2
\end{array}
\right)+(\Delta Y^2_\alpha)^2
\left(
\begin{array}{lcr}
0&0&0\\
0&1&2\\
0&2&4
\end{array}
\right).
$$
With this formula we can reason, trying to find conditions for the determinant of $\Gamma[Z]$ to be positive.
For instance if the L\'evy measures of  $Y^1$ and $Y^2$ are infinite, it follows that $Z_t$ has a density as soon as 
$$
\mbox{dim }\mathcal{L}\left\{\left(
\begin{array}{c}
1\\
2(Y^1_t-\Delta Y_\alpha^1)\\
(Y^1_t-\Delta Y_\alpha^1)
\end{array}\right), \left(\begin{array}{c}
0\\
1\\
2
\end{array}\right)\quad\alpha\in JT\right\}=3.$$
But  $Y^1$ possesses necessarily jumps of different sizes, hence  $Z_t$ has a density on $\mathbb{R}^3$.

It follows that the integro-differential operator
$$\tilde{A}f(z)=
\int\!\left[f(z)-f\!
\left(\begin{array}{c}
z_1+y_1\\
z_2+2z_1y_1+y_2\\
z_3+z_1y_1+2y_2
\end{array}\right)
-(f^\prime_1(z)\;f^\prime_2(z)\;f^\prime_3(z))\left(
\begin{array}{c}
y_1\\
2z_1y_1+y_2\\
z_1y_1+2y_2
\end{array}\right)
\right]\sigma(dy_1dy_2)
$$
is hypoelliptic at order zero, in the sense that its semigroup $P_t$ has a density. No minoration is supposed on the growth of the L\'evy measure near 0 as assumed by many authors. 

This result implies that for any L\'evy process $Y$ satisfying the above hypotheses, even a subordinated one in the sense of Bochner, the process $Z$ is never subordinated of the Markov process $X$ solution of equation (\ref{diffusion}) (otherwise it would live on the same manifold as the initial diffusion).
\section{Application to SDE's driven by a Poisson measure}
\subsection{The equation we study}
We consider another  probability space $(\Om_2,\cA_2 ,\bbP_2)$ on which an $\R^n$-valued semimartingale $Z=(Z^1 ,\cdots,Z^n)$ is defined, $n\in\Ne$. We adopt the following assumption  on the bracket of $Z$ and on the total variation of its finite variation part. It is satisfied if both are dominated by the Lebesgue measure uniformly:\\
 {\underline{Assumption on $Z$}}:
There exists a positive constant $C$ such that  for any square integrable $\R^n$-valued predictable process $h$:
\begin{equation}\label{CondCrochet} \forall t\geq 0 ,\ \E  [(\int_0^t h_s dZ_s )^2]\leq C^2 \E [\int_0^t |h_s|^2 ds].\end{equation}
We shall work on the product probability space:
$(\Om,\cA ,\bbP)=(\Om_1 \times \Om_2 ,\cA_1 \otimes \cA_2 ,\bbP_1
\times \bbP_2 ).$\\
For simplicity, we fix  a finite terminal time $T>0$.\\
Let $d\in\Ne$, we consider the following SDE:
\begin{equation}\label{eq}
X_t =x+\int_0^t \int_X c(s,X_{s^-},u)\tN (ds,du)+\int_0^t \sigma
(s,X_{s^-})dZ_s
\end{equation}
where $x\in\R^d$,  $c:\R^+
\times \R^d \times X\rightarrow\R^d$ and $\sigma:\R^+ \times \R^d \rightarrow\R^{d\times n}$ satisfy the set of hypotheses below denoted (R).\\

\underline{Hypotheses (R):}

\noindent 1. There exists $\eta \in L^2 (X,\nu)$  such that:

a) for all $t\in [0,T]$ and $u\in X$, $c(t,\cdot ,u)$ is
differentiable with continuous derivative and
\[\forall u\in X,\  \sup_{t\in [0,T], x\in\R^d} | D_x c(t,x
,u)|\leq \eta (u),\]
\indent b) $\forall (t,u)\in [0,T]\times X,\ |c(t,0,u)|\leq \eta (u)$,

c)  for all $t\in [0,T]$ and $x\in \R^d$, $c(t,x,\cdot )\in\bbd$
and
\[  \sup_{t\in [0,T], x\in\R^d} \gamma [c(t,x,\cdot)](u)\leq \eta^2 (u)
,\]
\indent d) for all $t\in [0,T] $, all $x\in \R^d$ and $u\in X$,\ the matrix $I+D_x
c(t,x,u)$ is invertible and
\[ \sup_{t\in [0,T], x\in \R^d} \left|\left(I+D_x c(t,x,u)\right)^{-1}\right|\leq \eta
(u).\]
2.  For all $t\in [0,T]$ , $\sigma(t,\cdot )$ is
differentiable with continuous derivative and
\[  \sup_{t\in [0,T], x\in\R^d} | D_x \sigma(t,x)|<+\infty.\]
3. As a consequence of  hypotheses 1. \!and 2. \!above, it is well known that equation \eqref{eq} admits a unique solution $X$ such that
$ \E [\sup_{t\in [0,T]} |X_t |^2 ]<+\infty$.
 We suppose that for all $t\in [0,T]$, the matrix $(I+\sum_{j=1}^n D_x \sigma_{\cdot ,j} (t, X_{t^-})\Delta Z_t^j )$ is invertible and its inverse is bounded by a deterministic constant uniformly with respect to $t\in [0,T]$.\\
 
\noindent {\it Remark:} We have defined a Dirichlet structure $(\bbD ,\cE )$ on  $L^2 (\Om_1 , \bbP_1 )$. Now, we work on the product space, $\Om_1 \times \Om_2$. Using  natural notations, we consider from now on that $(\bbD ,\cE )$ is a Dirichlet structure on $L^2 (\Om, \bbP)$. In fact, it is the product structure of $(\bbD ,\cE)$ with the trivial one on $L^2 (\Om_2 ,\bbP_2 )$ (see \cite{bouleau-hirsch2} ). Of course, all the properties remain true. In other words, we only differentiate w.r.t. the Poisson noise and not w.r.t.  the one introduced by $Z$.\\
\subsection{Spaces of processes and functional calculus}
We denote by $\cP$ the predictable sigma-field on
$[0,T]\times \Omega$ and we define the following sets of processes:
\begin{itemize}
\item $\cH$ : the set of real valued  processes $(X_t )_{t\in [0,T]}$, defined on
$(\Omega ,\cA,\bbP)$, which belong to $L^2 ([0,T]\times \Omega)$.
\item $\LP$ : the set of predictable processes in $\cH$.
\item $\LD$ : the set of real valued  processes $(H_t )_{t\in
[0,T]}$, which belong to $L^2 ([0,T]; \bbD)$ i.e. such that
\[ \| H\|^2_{\LD} =\E [ \int_0^T |H_t |^2 dt]+\int_0^T \cE (H_t )dt
<+\infty.\]
\item $\LDP$ : the subvector space of predictable processes in $\LD$.
\item $\LDdP$ : the set of real valued processes $H$ defined on $[0,T]\times
\Omega \times X$ which are predictable and belong to $L^2
([0,T];\bbD\otimes \bbd)$ i.e. such that
\[ \| H\|^2_{\LDdP} =\E [ \int_0^T \int_X |H_t |^2 \nu(du)dt]
+\int_0^T\int_X \cE (H_t(\cdot ,u)) \nu(du)dt+\E [ \int_0^T
e(H_t)dt] <+\infty.\]
\end{itemize}
The main idea is to differentiate equation \eqref{eq}, to do that we need some functional calculus. It is given by the next Proposition that we prove by approximation:
\begin{Pro}{\label{Estimees} } Let $H\in \LDdP$ and $G\in \LDP^n$, then:
\begin{enumerate}
\item The process
$$ \forall t\in [0,T], \ X_t =\int_0^t\int_X H(s,w,u)\tN (ds,du)$$
is a square integrable martingale which belongs to $\LD$ and such
that the process $X^- =(X_{t^-})_{t\in [0,T]}$ belongs to $\LDP$.
The gradient operator satisfies for all $t\in [0,T]$:
\begin{equation}\label{Derive1}X_t^\sharp (w,\hw)=\int_0^t\int_X H^\sharp
(s,w,u,\hw ) d\tN (ds,du)+\\\int_0^t\int_{X\times R} H^{\flat}
(s,w,u,r)N\odot\rho (ds,du,dr).\end{equation}
\item The process
\[ \forall t\in [0,T], \ Y_t =\int_0^t G(s,w)dZ_s\]
is a square integrable semimartingale which belongs to $\LD$,
$Y^-=(Y_{t^-})_{t\in [0,T]}$ belongs to $\LDP$ and
\begin{equation}\label{Derive2}\forall t\in [0,T],\ Y_t^\sharp (w,\hw)=\int_0^t G^\sharp
(s,w,\hw ) dZ_s.\end{equation}\end{enumerate}
\end{Pro}
\subsection{Computation of the Carr\'e du champ matrix of the solution}
Applying the standard functional calculus related to Dirichlet forms, the previous Proposition and  a Picard iteration argument, we obtain:
\begin{Pro} The equation {\rm(\ref{eq})} admits a unique solution $X$ in
$\LD^d$. Moreover, the gradient of $X$ satisfies:
\begin{eqnarray*}
X_t^\sharp &=&\int_0^t\int_X D_x c(s,X_{s-},u)\cdot
X^\sharp_{s-}\tN (ds,du)\\&&+\int_0^t\int_{X\times R}
c^{\flat}(s,X_{s-},u,r)N\odot\rho (ds,du,dr)\\&&+\int_0^t D_x
\sigma (s,X_{s-})\cdot X^\sharp_{s-} dZ_s .
\end{eqnarray*}
\end{Pro}
Let us define the  $\R^{d\times d}$-valued processes $U$ by
$$dU_s=\sum_{j=1}^nD_x\sigma_{.,j}(s,X_{s-})dZ_s^j,$$
and the derivative of the flow generated by $X$:
\begin{eqnarray*}
K_t &=& I+\int_0^t\int_X D_x c(s,X_{s-} ,u)K_{s-} \tN (ds ,du)
+\int_0^t dU_sK_{s-}.
\end{eqnarray*}
\begin{Pro}
 Under our
hypotheses, for all $t\geq 0$, the matrix $K_t$ is invertible and its inverse
 $\bK_t =(K_t )^{-1}$ satisfies:
\begin{eqnarray*}
\bK_t &=& I -\int_0^t\int_X \bK_{s-} (I+D_x c(s,X_{s-} ,u))^{-1} {
D_x c(s,X_{s-} ,u)}\tN (ds ,du)
\\&&-\int_0^t \bK_{s-} dU_s+\sum_{s\leq t}\bK_{s-}(\Delta U_s)^2(I+\Delta U_s)^{-1}\\&&+\int_0^t\bK_s d<U^c,U^c>_s.
\end{eqnarray*}
\end{Pro}
We are now able to calculate the carr\'e du champ matrix. This is the aim of the next Theorem, to show how simple is the {\it lent particle method} we give a sketch of the proof.
\begin{Th}{\label{OCC}}
For all $t\in [0,T]$,
\begin{eqnarray*}
\Gamma [X_t ]&=&K_t  \int_0^t \int_X\bK_{s} \gamma[c(s
,X_{s-} ,\cdot )]\bK_{s}^{\ast}\, N (ds, du)
K_t^{\ast}.
\end{eqnarray*}
\end{Th}
\begin{proof}
Let $(\alpha ,u)\in [0,T]\times X$. We put
$ X_t\alu =\crea_{(\alpha ,u)} X_t.$
\begin{eqnarray*}
X_t\alu &=&x+\int_0^{\alpha} \int_X c(s,X_{s^-}\alu,u')\tN
(ds,du')\\&&+\int_0^{\alpha} \sigma(s,X_{s^-}\alu) dZ_s+c(\alpha
,X_{\alpha^-}\alu ,u)\\&&+\int_{]\alpha, t]} \int_X
c(s,X_{s^-}\alu,u')\tN (ds,du')+\int_{]\alpha, t]} \sigma
(s,X_{s^-}\alu)dZ_s.
\end{eqnarray*}
Let us remark that $X_t\alu =X_t $ if $t<\alpha$ so that, taking
the gradient with respect to the variable $u$, we obtain:
\begin{eqnarray*}
(X_t\alu )^\flat&=&(c (\alpha ,X_{\alpha^-}\alu ,u))^\flat
\\&&+\int_{]\alpha, t]} \int_X D_x c(s,X_{s^-}\alu,u')\cdot
(\xs\alu)^\flat \tN (ds,du')\\&&+\int_{]\alpha, t]}
D_x\sigma(s,X_{s^-}\alu)\cdot (\xs\alu)^\flat dZ_s.
\end{eqnarray*}
 Let us now introduce
the process $K_t\alu = \crea_{(\alpha ,u)}(K_t)$ which satisfies
the following SDE:
$$
K_t\alu  = I+\int_0^t\int_X D_x c(s,\xs\alu ,u') K_{s-}\alu \tN (ds ,du')
+\int_0^t dU_s^{(\alpha,u)}K_{s-}^{(\alpha,u)}
$$
and its inverse $\bK_t\alu = (K_t\alu)^{-1}$.
Then, using the flow property,  we have:
\[\forall t\geq 0 ,\ (X_t\alu )^\flat = K_t\alu  \bK_{\alpha}\alu
(c(\alpha ,X_{\alpha^-} ,u))^\flat .\]
Now, we calculate the carr\'e
du champ and then we take back the particle:
\[\forall t\geq 0 ,\ \anni_{(\alpha ,u)}\gamma [ (X_t\alu )] = K_t  \bK_{\alpha}
\gamma[c(\alpha ,X_{\alpha^-} ,\cdot )]
\bK_{\alpha}^{\ast} K_t^{\ast}.\]
Finally integrating with respect to $N$ we get
\begin{eqnarray*}
\forall t\geq 0 ,\ \Gamma [X_t]& =& K_t  \int_0^t
\int_X\bK_{s} \gamma[c(s ,X_{s^-} ,\cdot )](u)
\bK_{s}^{\ast}N (ds, du) K_t^{\ast}.\end{eqnarray*}
\end{proof}
\subsection{First application: the regular case}
An immediate consequence of the previous Theorem is: 
 \begin{Pro} Assume that X is a topological space, that the intensity measure $ds\times\nu$ of $N$ is such that $\nu$ has an infinite mass near some point $u_0$ in $X$. If the matrix $(s,y,u)\rightarrow\gamma[c(s,y,\cdot)](u)$ is continuous on a neighborhood of $(0,x,u_0)$ and invertible at $(0,x,u_0)$, then the solution $X_t$ of {\rm(\ref{eq})} has a density for all $t\in ]0,T]$.
 \end{Pro}
\subsection{Application to SDE's driven by a L\'evy process}
Let  $Y$ be a L\'evy process with values in $\mathbb{R}^d$, independent of another variable $X_0$.\\
We consider the following equation
\[ X_t =X_0 +\int_0^t a (X_{s-},s)\; dY_s ,\ \ t\geq 0\]
where $a: \R^k \times \R^+ \rightarrow \R^{k\times d}$ is a given map.
\begin{Pro} We assume that:
\begin{enumerate}
\item  The L\'evy measure, $\nu$, of $Y$ satisfies hypotheses of the example given in Section \ref{example} with $\nu (O)=+\infty$ and $\xi_{i,j}(x)=x_i \delta_{i,j}$. Then we may choose the operator $\gamma$ to be
$$\gamma[f]=\frac{\psi (x)}{k(x)}\sum_{i=1}^d x_i^2\sum_{i=1}^d(\partial_if)^2\quad \mbox{for }f\in\mathcal{C}^{\infty}_0(\mathbb{R}^d).$$
\item $a$ is $\mathcal{C}^1\cap Lip$ with respect to the first variable uniformly in $s$ and
$$\sup_{t,x}|(I+D_xa\cdot u)^{-1}(x,t)|\leq \eta (u),$$where $\eta \in L^2 (\nu )$.
\item $a$ is continuous with respect to the second variable at 0, and such that the matrix $aa^\ast(X_0,0)$ is invertible;
\end{enumerate}
then for all $t>0$ the law of $X_t$ is absolutely continuous w.r.t. the Lebesgue measure.
\end{Pro}
\begin{proof} We just give an idea of the proof in the case $d=1$:\\
Let us recall that
 $\gamma [f]=\displaystyle\frac{\psi (x)}{k(x)}x^2 f'^2(x)$.\\
We have the representation: $ Y_t =\int_0^t \int_{\R} u\tN (ds ,du),$
so that 
\[ X_t =X_0 +\int_0^t \int_{\R} a(s,X_{s-})u\; \tN (ds,du).\]
The lent particle method yields:
\begin{eqnarray*}
\Gamma [X_t]&=&K_t^2  \int_0^t
\int_X\bK_{s}^2 a^2(s, X_{s-})\gamma[j](u)
N (ds, du) 
\end{eqnarray*}
where $j$ is the identity application: $\gamma[j](u)=\displaystyle\frac{\psi (u)}{k(u)}u^2$.\\
So
\begin{eqnarray*}
\Gamma [X_t]&=&K_t^2  \int_0^t
\int_X\bK_{s}^2 a^2(s, X_{s-})\displaystyle\frac{\psi (u)}{k(u)}u^2
N (ds, du)\\
&=&K_t^2 \sum_{\alpha<t} \bK_{s}^2 a^2(s, X_{s-})\displaystyle\frac{\psi (\Delta Y_s)}{k(\Delta Y_s)}\Delta Y_s^2,
\end{eqnarray*}
and it is easy to conclude.
\end{proof}
{\it Remarks:}\\
(i) We refer to \cite{bouleau-denis2} for other examples and applications.\\
(ii) Let us finally remark that as easily seen,  one can iterate the gradient and so obtain criteria of regularity for the density of Poisson functionals such as solutions of SDE's, this is the object of a forthcoming paper.

 Ecole des Ponts,\\
ParisTech, Paris-Est\\
6 Avenue Blaise Pascal\\ 77455 Marne-La-Vallée Cedex 2
FRANCE\\bouleau@enpc.fr \\  \\ Equipe Analyse et Probabilités,
\\Universit\'{e} d'Evry-Val-d'Essonne,\\Boulevard François Mitterrand\\
91025 EVRY Cedex FRANCE\\ldenis@univ-evry.fr

\begin{thebibliography}{00}
\bibitem{bichteler-gravereaux-jacod}{\sc Bichteler K., Gravereaux J.-B., Jacod J.} {\it Malliavin Calculus for Processes with Jumps} (1987)
\bibitem{bouleau3}{\sc Bouleau N.} {\it Error Calculus for Finance and Physics, the Language of Dirichlet Forms}, De Gruyter (2003).
\bibitem{bouleau4}{\sc Bouleau N.} "Error calculus and regularity of Poisson functionals: the lent particle method" C. R. Acad. Sc. Paris,  Math\'ematiques, 
Vol 346, n13-14, p. 779-782, (2008).
 \bibitem{bouleau-denis1} {\sc Bouleau N.} and {\sc Denis L.} ``Energy image density property and the lent particle method for Poisson measures" {\it Jour. of Functional Analysis} 257, p. 1144-1174,  (2009).
 \bibitem{bouleau-denis2} {\sc Bouleau N.} and {\sc Denis L.} ``Application of the lent particle method to Poisson driven SDE's", in revision in Probability Theory and Related Fields.
\bibitem{bouleau-hirsch1}{\sc Bouleau N.} and {\sc Hirsch F.}"Formes de Dirichlet g\'en\'erales et densit\'e des variables al\'eatoires r\'eelles sur l'espace de Wiener" {\it J. Funct. Analysis} 69, 2, p. 229-259,  (1986).
\bibitem{bouleau-hirsch2}{\sc Bouleau N.} and {\sc Hirsch F.} {\it Dirichlet Forms and Analysis on Wiener Space} De Gruyter (1991).
\bibitem{coquio}{\sc Coquio A.} "Formes de Dirichlet sur l'espace canonique de Poisson et application aux \'equations diff\'erentielles stochastiques" {\it Ann. Inst. Henri Poincar\'e} vol 19, n1, p. 1-36, (1993)
\bibitem{denis}{\sc Denis L.} "A criterion of density for
solutions of Poisson-driven SDEs" {\it Probab. Theory Relat.
Fields} 118, p. 406-426,  (2000).
\bibitem{fukushima-oshima-takeda}{\sc Fukushima M., Oshima Y.} and {\sc Takeda M.} {\it Dirichlet Forms and Symmetric Markov Processes} De Gruyter (1994).
\bibitem{ikeda-watanabe}{\sc Ikeda N., Watanabe S.} {\it Stochastic Differential Equation and Diffusion Processes}, North-Holland, Koshanda (1981).
\bibitem{ishikawa-kunita}{\sc Ishikawa Y.} and {\sc Kunita H.} "Malliavin calculus on the Wiener-Poisson space and its application to canonical SDE with jumps" {\it Stoch. Processes and their App.} 116, p. 1743-1769, (2006).
\bibitem{ma-rockner2}{\sc Ma} and {\sc R\"ockner M.} "Construction of diffusion on configuration spaces" {\it Osaka J. Math.} 37, p. 273-314,  (2000).
\bibitem{nualart-vives}{\sc Nualart D.} and {\sc Vives J.} "Anticipative calculus for the Poisson process based on the Fock space", {\it S\'em. Prob. XXIV}, Lect. Notes in M. 1426, Springer (1990).
\bibitem{picard}{\sc Picard J.}"On the existence of smooth densities for jump processes" {\it Probab. Theorie Relat. Fields} 105, p. 481-511, (1996).
\bibitem{privault} {\sc Privault N.} "A pointwise equivalence of gradients on configuration spaces", {\it C. Rendus Acad. Sc. Paris}, 327, 7, p. 677-682, (1998).
\bibitem{privault2}{\sc Privault N.} ``Equivalence of gradients on configuration spaces" {\it Random Oper. and Stoch. Equ.} Vol. 7, No. 3,  p. 241-262, (1999).
\bibitem{scotti} {\sc Scotti S.} {\it Applications de la Th\'eorie des Erreurs par Formes de Dirichlet}, Thesis Univ. Paris-Est, Scuola Normale Pisa, 2008. (http://pastel.paristech.org/4501/)
\bibitem{song}{\sc Song Sh.} "Admissible vectors and their associated Dirichlet forms" {\it Potential Analysis} 1, 4, p. 319-336, (1992).
\end{thebibliography}
 \end{document}